\documentclass{article}
\usepackage{amsthm}
\usepackage{amssymb}
\usepackage{amsfonts}
\begin{document}

\newtheorem{theorem}{Theorem}[section]
\newtheorem{definition}{Definition}[section]
\newtheorem{corollary}[theorem]{Corollary}
\newtheorem{lemma}[theorem]{Lemma}
\newtheorem{proposition}[theorem]{Proposition}
\newtheorem{step}[theorem]{Step}
\newtheorem{example}[theorem]{Example}
\newtheorem{remark}[theorem]{Remark}

\font\sixbb=msbm6
\font\eightbb=msbm8
\font\twelvebb=msbm10 scaled 1095
\newfam\bbfam
\textfont\bbfam=\twelvebb \scriptfont\bbfam=\eightbb
                           \scriptscriptfont\bbfam=\sixbb

\newcommand{\tr}{{\rm tr \,}}
\newcommand{\linspan}{{\rm span\,}}
\newcommand{\rank}{{\rm rank\,}}
\newcommand{\diag}{{\rm Diag\,}}
\newcommand{\Image}{{\rm Im\,}}
\newcommand{\Ker}{{\rm Ker\,}}

\def\bb{\fam\bbfam\twelvebb}
\newcommand{\enp}{\begin{flushright} $\Box$ \end{flushright}}
\def\cD{{\mathcal{D}}}
\def\N{\bb N}

\title{Maps on Grassmann spaces preserving the minimal principal angle
\thanks{The author was supported by grants J1-2454, and P1-0288 from ARRS, Slovenia.}}
\author
{Peter \v Semrl\footnote{Institute of Mathematics, Physics, and Mechanics, Jadranska 19, SI-1000 Ljubljana, Slovenia, peter.semrl@fmf.uni-lj.si} 
       }

\date{}
\maketitle

\begin{abstract}
Let $n$ be a positive integer and $H$ a Hilbert space. The description of the general form of bijective maps on the set of $n$-dimensional subspaces of $H$ preserving the maximal principal angle has been obtained recently. This is a generalization of Wigner's unitary-antiunitary theorem. In this paper we will obtain another extension of Wigner's theorem in which the maximal principal angle is replaced by the minimal one. Moreover, in this case we do not need the bijectivity assumption.
\end{abstract}
\maketitle




\section{Introduction and statement of the main results}

Let $H$ be a Hilbert space and $n$ a positive integer. We denote by $P_n (H)$ the set of all rank $n$ projections on $H$. In the case when $H$  
is an infinite-dimensional separable Hilbert space, the symbol $P_\infty (H)$ stands for the set of all projections whose image and kernel are both infinite-dimensional. By $\| \cdot \|$
we denote the usual operator norm on $B(H)$, the set of all bounded linear operators on $H$. The distance on $P(H)$, the set of all projections on $H$, induced by the
operator norm is usually called the gap metric.

Throughout this paper we will always identify projections of rank $n$ with $n$-dimensional subspaces:  to each $P \in P_n (H)$ we associate the image ${\rm Im}\, P$ of the projection $P$.
If we take two subspaces $P,Q \in P_n (H)$ then the relative subspace position of $P$ and $Q$ can be completely described by the set of $n$ principal angles.
The notion of principal angles was first investigated by Jordan, and has a wide range of applications in mathematical statistics, geometry, etc. 
The sequence of $n$ principal angles $0 \le \theta_1 \le \theta_2 \le \ldots \le \theta_n \le \pi / 2$ can be defined in various ways. We will present here a recursive definition. For any two vectors $x,y \in H$ of norm one we define the angle
between $x$ and $y$ by
$$
\angle (x,y) = \arccos | \langle x, y \rangle | \in \left[ 0, {\pi \over 2} \right].
$$
Then the first principal angle is defined by
$$
\theta_ 1 = \min \{ \angle (x,y) \, : \, x \in {\rm Im}\, P , \, y \in {\rm Im}\, Q , \, \| x \| = \| y \| = 1 \} = \angle (x_1 , y_1 ).
$$
The unit vectors $x_1$ and $y_1$ are the corresponding principal vectors (note that they are not uniquely determined). 
The other principal angles and vectors are then defined recursively by
$$
\theta_ k  = \min \{ \angle (x,y) \, : \, x \in {\rm Im}\, P , \, y \in {\rm Im}\, Q , \, \| x \| = \| y \| = 1, $$ $$ x \perp x_j \, {\rm and} \, y \perp y_j \ \, \forall j \in \{ 1, \ldots , k-1\}  \} = \angle (x_k , y_k ).
$$
We recall that the sines of the non-zero principal angles are exactly the non-zero singular values of the operator $P-Q$, each of them counted twice (see e.g.  (ii) of \cite[Theorem 26]{Gal}). 
In particular, $\|P-Q\|$ is the sine of the largest principal angle. 

A linear isometry $U : H \to H$ is a linear map satisfying $\langle Ux, Uy \rangle = \langle x,y \rangle$, $x,y \in H$, and a conjugate-linear  isometry $U : H \to H$ is a 
conjugate-linear map satisfying $\langle Ux, Uy \rangle = \langle y,x \rangle$, $x,y \in H$. A surjective linear isometry is a unitary operator, while a surjective
conjugate-linear isometry is called an antiunitary operator. 

In mathematical foundations of quantum mechanics the Grassmann space $P_1(H)$ is used to represent the set of pure states of a quantum system.
The famous Wigner's unitary-antiunitary theorem describes the general form of transformations of the set of all pure states which preserve the transition probability.
Recall  that when we represent pure states with projections of rank one, then the transition probability between the states $P$ and $Q$ is $\tr (PQ)$, the trace of the product $PQ$.
It is an elementary linear algebra exercise to show that $\|P-Q\| = \sqrt{1 - \tr (PQ)}$ holds true for every pair $P,Q \in P_1 (H)$. Hence, Wigner's theorem can be interpreted either as the structural result for
isometries of $P_1 (H)$ or as the structural result for transformations on $P_1(H)$ preserving the (unique principal) angle between elements of $P_1(H)$.
The exact formulation reads as follows. Assume that $\phi : P_1(H) \to P_1 (H)$, $\dim H \ge 2$, is a map satisfying $\| \phi (P) - \phi (Q) \| = \| P - Q \|$, $P, Q \in P_1 (H)$. Then there
exists a linear or conjugate-linear isometry $U : H \to H$ such that
\begin{equation}\label{stan}
\phi (P) = UPU^\ast
\end{equation}
for all $P \in P_1 (H)$. Every map of the form (\ref{stan}) will be called a standard map.  Note that under the additional assumption that $\phi$ is surjective, the operator $U$ appearing
in (\ref{stan}) is either a unitary or an antiunitary operator.

In \cite{BJM}, 
Botelho,  Jamison, and  Moln\'ar described the general form of surjective isometries of $P_n(H)$ under some dimensionality constraints.
Geh\' er and the author of the present paper
succeeded to extend this result to all possible dimensions \cite{GeS}. 
They proved that if  $\dim H > n$ and $\dim H \neq 2n$, then  
every surjective isometry $\phi\colon P_n (H) \to P_n (H)$ is standard.
In the case when $\dim H = 2n$,  we have the additional possibility that $\phi$ is a standard map composed with the orthogonal complementation $P \mapsto I-P$, $P\in P_n (H)$.
In their recent paper \cite{GeS2} they solved the remaining case of $P_\infty (H)$. They assumed that $H$ is an infinite-dimensional separable Hilbert 
space and $\phi\colon P_\infty (H) \to P_\infty (H)$ a surjective isometry and proved that then $\phi$ must be either a standard map or 
 a standard map composed by the orthogonal complementation.
On one hand, one can consider all these results as far reaching generalizations of the classical Wigner's theorem, but on the other hand, one can observe an essential difference: the improved version of 
Wigner's theorem formulated above describes the general form of not necessarily surjective isometries of $P_1 (H)$, while the proofs of the results mentioned in this paragraph depend heavily on the surjectivity
assumption. 

So far, two generalizations of Wigner's theorem have been obtained without using the surjectivity assumption but under different preserving properties involving all principal angles, not just the ``largest" one as in the 
previous paragraph (note that we put quotation marks because of $P_\infty (H)$ - in this case we can not speak of the maximal angle between unit vectors from two subspaces, but rather of the supremum of all possible angles).
In \cite{Mo} and \cite{Mo2},  Moln\'ar characterized (not necessarily surjective) transformations of $P_n(H)$ which preserve the complete system of principal angles. 
Again, these transformations are standard except in the case when $\dim H=2n$ and in this special case we have the additional possibility of a standard map composed by the orthogonal complementation.
This theorem has been improved by Geh\' er \cite{Geh} who considered (not necessarily surjective) transformations of $P_n(H)$ which preserve the  ``transition probability" between projections $P, Q \in P_n (H)$
defined as the trace of the product $PQ$. He showed that all such maps are standard except in the case when $\dim H = 2n$ where we have the usual additional possibility. It should be mentioned here
that for any pair of projections $P,Q \in P_n (H)$ we have
$$
\tr (PQ) = \sum_{k=1}^n \cos^2 \theta_k ,
$$
where $\theta_1 , \ldots , \theta_n$ is the complete system of principal angles between $P$ and $Q$.

In the above two paragraphs we have described several generalizations of Wigner's theorem that serve as a motivation for our research. We will deal with non-surjective generalizations of Wigner's theorem
(as in the second paragraph) preserving a single principal angle as in the first paragraph - but this time we will be interested in the ``smallest" principal angle, not the ``largest" one. 

We start with the following trivial observation.

\begin{proposition}\label{lako}
Let $H$ be an infinite-dimensional separable Hilbert space and $\phi : P_\infty (H) \to P_\infty (H)$ a surjective map such that
$$
\inf   \{ \angle (x,y) \, : \, x \in {\rm Im}\, P , \, y \in {\rm Im}\, Q , \, \| x \| = \| y \| = 1 \} $$ $$= 
\inf   \{ \angle (x,y) \, : \, x \in {\rm Im}\, \phi (P) , \, y \in {\rm Im}\, \phi (Q) , \, \| x \| = \| y \| = 1 \}
$$
for every $P,Q \in P_\infty (H)$. Then there exists a unitary or antiunitary operator $U : H \to H$ such that
$$
\phi (P) = UPU^\ast
$$
for all $P \in P_\infty (H)$.
\end{proposition}

The proof of this statement is very easy. All one needs to do is to verify that $\phi$ is actually bijective and to observe that under our assumptions $\phi$ preserves orthogonality of projections. Then the
result follows directly from the known structural result on the orthogonality preservers on $P_\infty (H)$ \cite{Sem}. We stated this trivial result to emphasize an interesting difference between $P_n (H)$ and $P_\infty (H)$.
Namely, while the assumption of surjectivity is essential in the above result we can get the desired structural result for preservers of minimal principal angle on $P_n(H)$ without this assumption.
More precisely, we can prove the following two statements.

\begin{proposition}\label{manje}
Let $H$ be an infinite-dimensional separable Hilbert space. Then there exists a non-standard  map $\phi : P_\infty (H) \to P_\infty (H)$ satisfying
$$
\inf   \{ \angle (x,y) \, : \, x \in {\rm Im}\, P , \, y \in {\rm Im}\, Q , \, \| x \| = \| y \| = 1 \} $$ $$= 
\inf   \{ \angle (x,y) \, : \, x \in {\rm Im}\, \phi (P) , \, y \in {\rm Im}\, \phi (Q) , \, \| x \| = \| y \| = 1 \}
$$
for every $P,Q \in P_\infty (H)$.
\end{proposition}

Before formulating our main theorem we introduce the notion of the minimal angle between two finite rank projections.

\begin{definition}\label{zak}
Let $P,Q \in P_n(H)$. Then the minimal angle between $P$ and $Q$ is defined  to be the first principal angle, that is,
$$
{\rm ma}\, (P,Q) = \min  \{ \angle (x,y) \, : \, x \in {\rm Im}\, P , \, y \in {\rm Im}\, Q , \, \| x \| = \| y \| = 1 \} \in \left[ 0, {\pi \over 2} \right].
$$
\end{definition}

\begin{theorem}\label{main}
Let $H$ be a Hilbert space, $n$ a positive integer such that $\dim H \ge 2n$, and $\phi : P_n (H) \to P_n (H)$ a map satisfying
$$
{\rm ma}\, (\phi (P ), \phi (Q)) = {\rm ma}\, (P,Q)
$$ 
for all $P,Q \in P_n (H)$. Then there exists a linear or conjugate-linear isometry $U : H \to H$ such that
$$
\phi (P) = UPU^\ast
$$
for every $P \in P_n (H)$, or $\dim H = 2n$ and
$$
\phi (P) = U(I-P)U^\ast
$$
for every $P \in P_n (H)$.
\end{theorem}

In the special case when $n=1$ and $\dim H = 2$ every map $\phi : P_1 (H) \to P_1 (H)$ satisfying ${\rm ma}\, (\phi (P ), \phi (Q)) = {\rm ma}\, (P,Q)$, 
$P,Q \in P_n (H)$, is of the form $\phi (P) = UPU^\ast$,
$P \in P_n (H)$, for some unitary or antiunitary operator $U: H \to H$. In other words, in this special case the second possibility in the conclusion of Theorem \ref{main} is not needed. To see this we identify $P_1 (H)$ with the set of all $2\times2$ hermitian idempotent matrices. Then every element of $P_1(H)$ is a matrix of the form
$$
\left[ \begin{matrix} {p & z\sqrt{p (1-p)} \cr \overline{z}\sqrt{p (1-p)} & 1-p \cr} \end{matrix} \right]
$$
for some real $p$, $0 \le p \le 1$, and some complex number $z$ of modulus one. It is straightforward to verify that
$$
I - P = UP^t U^\ast , \ \ \ P \in P_1 (H),
$$
where
$$
U = \left[ \begin{matrix} {0 & 1 \cr -1 & 0 \cr} \end{matrix} \right] .
$$
Note that $P^t = \overline {P}$, $P \in P_1 (H)$, where $\overline {P}$ is the matrix obtained from $P$ by applying the complex-conjugation entrywise.

It is easy to check that when $\dim H = 2n$ and $n > 1$, the map $P \mapsto I-P$ differs from every map $P \mapsto UPU^\ast$, where $U : H \to H$ is any unitary or antiunitary operator. Indeed, assume that $I-P = UPU^\ast$, $P \in P_n (H)$, for some unitary or antiunitary operator $U$. 
We will identify linear operators on $H$ with matrices. Let $P_1 , P_2 , P_3$ be diagonal projections of rank $n$, that is, $(2n) \times (2n)$ diagonal matrices with $n$ diagonal entries equal to $1$ and the other $n$ diagoanl entries zero. We choose $P_1 , P_2 , P_3$ in such a way that the first diagonal entry of each of them is equal to $1$ and for each $j$, $2 \le j \le 2n$, the $j$-th diagonal entry of at least one of these three matrices is nonzero. Set $A = (1/3) (P_1 + P_2 + P_3)$. Then the first diagonal entry of $A$ equals one and all other diagonal entries are nonzero, meaning that $1$ is an eigenvalue of $A$ and $A$ is invertible.
We have
$$
I - A = I - {1 \over 3} (P_1 + P_2 + P_3)  =  {1 \over 3} ((I -P_1) + (I- P_2) + (I - P_3))  $$ $$=  {1 \over 3} (UP_1 U^\ast + UP_2 U^\ast + UP_3 U^\ast)
= U  \, \left( {1 \over 3} (P_1 + P_2 + P_3)\right) \, U^\ast = UAU^\ast.
$$
Because $A$ is invertible the operator $UAU^\ast$ is invertible, while the fact that $1$ is an eigenvalue of $A$ yields that $I - A$ is singular. This contradicts the above equation.

Nothing can be said about maps  $\phi : P_n (H) \to P_n (H)$ satisfying
${\rm ma}\, (\phi (P ), \phi (Q)) = {\rm ma}\, (P,Q)$, $P,Q \in P_n (H)$, when $n+1 \le \dim H < 2n$.  Namely, in this case the intersection of any two $n$-dimensional subspaces is a non-zero
subspace, and thus, $ {\rm ma}\, (P,Q) = 0$ for every pair $P,Q \in P_n (H)$. It follows that in this case each map $\phi : P_n (H) \to P_n (H)$ satisfies the condition 
${\rm ma}\, (\phi (P ), \phi (Q)) = {\rm ma}\, (P,Q)$, $P,Q \in P_n (H)$.

\section{Preliminary results}

We start this section with an easy observation. Recall that two projections $P$ and $Q$ are orthogonal, $P \perp Q$, if and only if $PQ=0$ which is equivalent to $QP=0$. 
This is further equivalent to ${\rm Im}\, P \perp {\rm Im}\, Q$. If $P$ and $Q$ are projections of the same finite rank, then we have $P\perp Q$ if and only if ${\rm ma}\, (P,Q) = {\pi \over 2}$.
Of course, if $P,Q \in P_\infty (H)$, then $P \perp Q$ if and only if
$$
\inf   \{ \angle (x,y) \, : \, x \in {\rm Im}\, P , \, y \in {\rm Im}\, Q , \, \| x \| = \| y \| = 1 \} = {\pi \over 2} .
$$

Let $P,Q$ be any non-zero projections (not necessarily of the same rank and not necessarily of finite rank) on $H$. We will extend Definition \ref{zak} to any such pair of projections by
$$
{\rm ma}\, (P,Q) = \inf  \{ \angle (x,y) \, : \, x \in {\rm Im}\, P , \, y \in {\rm Im}\, Q , \, \| x \| = \| y \| = 1 \} \in \left[ 0, {\pi \over 2} \right].
$$

\begin{lemma}\label{padejmo}
Let $H$ be a Hilbert space and $P_1, P_2, Q_1 , Q_2$ non-zero projections on $H$. Let further $P$ and $Q$ be projections on the orthogonal direct sum $H \oplus H$ with corresponding matrix
representations
$$
P = \left[ \begin{matrix}  { P_1 & 0 \cr 0 & P_2 \cr } \end{matrix} \right] \ \ \ {\rm and} \ \ \ Q = \left[ \begin{matrix}  { Q_1 & 0 \cr 0 & Q_2 \cr } \end{matrix} \right].
$$
Then
$$
{\rm ma}\, (P,Q) = \min \{ {\rm ma}\, (P_1 ,Q_1 ) , {\rm ma}\, (P_2 ,Q_2 ) \}.
$$
\end{lemma}

\noindent
{\sl Proof.} 
Let $K$ be any Hilbert space. We define a partial order on the set of all projections on $K$ by
$R \le S$ if and only if ${\rm Im}\, R \subset {\rm Im}\, S$, $R,S \in P(K)$. Of course, this condition is equivalent to any, and hence both, of the equalities $RS = R$ and $SR = R$.
It follows directly from the definition of ${\rm ma}\, ( \cdot , \cdot)$ that if $0\not= R_1 \le S_1$ and 
$0\not= R_2 \le S_2$ for some $R_1 , R_2 , S_1 , S_2 \in P(K)$, then
$$
{\rm ma}\, (R_1 , R_2 ) \ge {\rm ma}\, (S_1 , S_2 ).
$$

Now,
$$
{\rm ma}\, (P_1,Q_1) = {\rm ma}\, \left( 
 \left[ \begin{matrix}  { P_1 & 0 \cr 0 & 0 \cr } \end{matrix} \right] ,  \left[ \begin{matrix}  { Q_1 & 0 \cr 0 & 0 \cr } \end{matrix} \right] \right) ,
$$
and since 
$$
 \left[ \begin{matrix}  { P_1 & 0 \cr 0 & 0 \cr } \end{matrix} \right] \le P \ \, {\rm and} \ \, 
 \left[ \begin{matrix}  { Q_1 & 0 \cr 0 & 0 \cr } \end{matrix} \right] \le Q
$$
we have
$$
{\rm ma}\, (P_1,Q_1) \ge {\rm ma}\, (P,Q) .
$$
Of course, the same is true if we replace $P_1$ and $Q_1$ by $P_2$ and $Q_2$, respectively, and therefore,
$$
{\rm ma}\, (P,Q) \le \min \{ {\rm ma}\, (P_1 ,Q_1 ) , {\rm ma}\, (P_2 ,Q_2 ) \}.
$$

To prove the reverse inequality we choose an arbitrary positive real number $\varepsilon$. Then there exist unit vectors
$$
 \left[ \begin{matrix}  { x_0 \cr y_0 \cr } \end{matrix} \right] \in {\rm Im}\, P \ \, {\rm and} \ \, 
 \left[ \begin{matrix}  { u_0 \cr v_0 \cr } \end{matrix} \right] \in {\rm Im}\, Q, 
$$
that is, $x_0 \in {\rm Im}\, P_1$, $y_0 \in {\rm Im}\, P_2$, 
$u_0 \in {\rm Im}\, Q_1$, and $v_0 \in {\rm Im}\, Q_2$,  such that
$$
\angle \left(  \left[ \begin{matrix}  { x_0 \cr y_0 \cr } \end{matrix} \right] , 
 \left[ \begin{matrix}  { u_0 \cr v_0 \cr } \end{matrix} \right] \right) < {\rm ma}\, (P,Q) + \varepsilon.
$$
After taking a small enough perturbation  of the above vectors we can assume that we have $x_0 \not= 0$, $y_0 \not=0$, $u_0 \not=0$, and $v_0 \not=0$
in the above inequality.
We further know that
$$
\angle \left(  \left[ \begin{matrix}  { x_0 \cr y_0 \cr } \end{matrix} \right] , 
 \left[ \begin{matrix}  { u_0 \cr v_0 \cr } \end{matrix} \right] \right) \ge {\rm ma}\, \left( \left[ \begin{matrix}  { M_1 & 0 \cr 0 & M_2 \cr } \end{matrix} \right]\  ,
 \left[ \begin{matrix}  { N_1 & 0 \cr 0 & N_2 \cr } \end{matrix} \right] \right),
$$
where $M_1 , M_2, N_1, N_2$ are the rank one projections on $H$ statisfying $M_1 x_0 = x_0$, $M_2 y_0 = y_0$, $N_1 u_0 = u_0$, and $N_2 v_0 = v_0$. In particular,
$M_1 \le P_1$, $M_2 \le P_2$, $N_1 \le Q_1$, and $N_2 \le Q_2$.

Now, ${\rm ma}\, (M,N)$, where
$$
M=  \left[ \begin{matrix}  { M_1 & 0 \cr 0 & M_2 \cr } \end{matrix} \right] \ \ \ {\rm and} 
\ \ \
N=  \left[ \begin{matrix}  { N_1 & 0 \cr 0 & N_2 \cr } \end{matrix} \right],
$$
is  the first principal angle between ${\rm Im}\, M$ and ${\rm Im}\, N$ and using the facts that the sine of the minimal non-zero principal angle
is the smallest non-zero singular value of $M-N$ and that the set of non-zero singular values of $M-N$ is the union of the set of
 non-zero singular values of $M_1 -N_1$ and  the set of
 non-zero singular values of $M_2 -N_2$, we see that 
$$
{\rm ma}\, (P,Q) + \varepsilon \ge
{\rm ma}\, (M,N) = \min \{ {\rm ma}\, (M_1 , N_1) , {\rm ma}\, (M_2 , N_2) \} $$ $$ \ge
 \min \{ {\rm ma}\, (P_1 , Q_1) , {\rm ma}\, (P_2 , Q_2) \}. 
$$
Since $\varepsilon$ was an arbitrary positive real number this completes the proof.
\enp

In the rest of this section we will assume that $H$ is an infinite-dimensional Hilbert space and $P_n (H)$, $n \ge 2$, will be considered as the set of all $n$-dimensional subspaces of $H$. For any vectors $x_1, \ldots , x_k \in H$ we will denote by $[x_1 , \ldots , x_k]$ the linear span of $x_1, \ldots , x_k$. We will say that two subspaces $U,V \in P_n (H)$ are 1-orthogonal, $ U \sharp V$, if there exists an orthonormal system of vectors $e_1, \ldots , e_{2n-1} \in H$ such that
$$
U = [e_1, \ldots , e_n] \ \ \ {\rm and} \ \ \ V = [e_n , \ldots , e_{2n-1}].
$$
In other words, we have
 $U \sharp V$ if and only if $ \dim (U \cap V) = 1$ and the orthogonal complement of $U \cap V$ in $U$ is an $(n-1)$-dimensional subspace of $H$ that is orthogonal to the orthogonal complement of $U \cap V$ in $V$.

\begin{lemma}\label{simulant}
Let $\phi : P_n (H) \to P_n (H)$ be a map such that for every pair $U,V \in P_n (H)$ we have
$$
U \perp V \iff \phi (U) \perp \phi (V)
$$
\begin{equation}\label{jotka}
U \cap V = \{ 0 \} \iff \phi (U) \cap \phi (V)  = \{ 0 \}.
\end{equation}
Then
\begin{itemize}
\item for every pair $U,V \in P_n (H)$ we have $U \sharp V \Rightarrow \phi (U) \sharp \phi (V)$, and
\item for every $U,V,W$ such that
\begin{equation}\label{uno}
U \sharp V \ \ \ {\rm and} \ \ \ U \sharp W \ \ \ {\rm and} \ \ \ V \sharp W
\end{equation}
and
\begin{equation}\label{duo}
U \cap V = U \cap W = V \cap W
\end{equation}
we have
$$
\phi (U) \sharp \phi (V) \ \ \ {\rm and} \ \ \ \phi (U ) \sharp \phi (W ) \ \ \ {\rm and} \ \ \ \phi (V ) \sharp \phi (W )
$$
and
$$
\phi (U ) \cap \phi (V ) = \phi (U ) \cap \phi (W)  = \phi (V ) \cap \phi (W).
$$
\end{itemize}
\end{lemma}

\begin{proof}
It is enough to prove the second claim because the first one is a straightforward consequence. So, let us assume that $U,V,W \in P_n (H)$ are subspaces satisfying (\ref{uno}) and (\ref{duo}). Let $e \in H$ be a unit vector such that
$$
U \cap V = U \cap W = V \cap W = [e]
$$
and denote by $U_1 , V_1 , W_1$ the orthogonal complements of $[e]$ in $U,V,W$, respectively. It follows from (\ref{uno}) that
$$
U_1 \perp V_1 \ \ \ {\rm and} \ \ \ U_1 \perp W_1 \ \ \ {\rm and} \ \ \ V_1 \perp W_1 .
$$
Consequently, we have
$$
U = [e, f_1, \ldots , f_{n-1}],
$$
$$
V = [e, f'_1, \ldots , f'_{n-1}],
$$
and
$$
W = [e, f''_1, \ldots , f''_{n-1}]
$$
for some orthonormal set of vectors $e,  f_1, \ldots , f_{n-1},  f'_1, \ldots , f'_{n-1},  f''_1, \ldots , f''_{n-1}$.

Since $H$ is infinite-dimensional we can find pairwise orthogonal subspaces $L, M_1 ,\ldots, M_{n-1}, M'_1 ,\ldots, M'_{n-1}, M''_1 ,\ldots, M''_{n-1} \in P_n (H)$ such that 
$$
e \in L \ \ \ {\rm and} \ \ \ f_j \in M_j \ \ \ {\rm and}   \ \ \ f'_j \in M'_j \ \ \ {\rm and}  \ \ \ f''_j \in M''_j , \ \ \ j= 1, \ldots, n-1. 
$$
Then the subspaces 
\begin{equation}\label{furi}
\phi (L), \phi (M_1 ),\ldots, \phi(M_{n-1}), \phi(M'_1 ),\ldots, \phi( M'_{n-1}), \phi(M''_1) ,\ldots, \phi( M''_{n-1})
\end{equation} 
are pairwise orthogonal, too.

Using (\ref{jotka}) we see that we can find unit vectors
$$
u \in \phi (U) \cap \phi (L),  v_1 \in \phi (U) \cap \phi (M_1), \ldots, v_{n-1} \in \phi (U) \cap \phi (M_{n-1})
$$
and because the vectors $u, v_1 , \ldots, v_{n-1}$ are pairwise orthogonal and $\phi (U)$ is of dimension $n$, we have
$$
\phi (U) = [u, v_1 , \ldots , v_{n-1}].
$$

In the same way we see that there are unit vectors 
$$
u' \in \phi (V) \cap \phi (L),  v'_1 \in \phi (V) \cap \phi (M'_1), \ldots, v'_{n-1} \in \phi (V) \cap \phi (M'_{n-1})
$$
and
$$
u'' \in \phi (W) \cap \phi (L),  v''_1 \in \phi (W) \cap \phi (M''_1), \ldots, v''_{n-1} \in \phi (W) \cap \phi (M''_{n-1})
$$ 
such that
$$
\phi (V) = [u', v'_1 , \ldots , v'_{n-1}]
$$
and
$$
\phi (W) = [u'', v''_1 , \ldots , v''_{n-1}].
$$
Because the subspaces (\ref{furi}) are pairwise orthogonal the vectors
$$
v_1 , \ldots , v_{n-1}, v'_1 , \ldots , v'_{n-1}, v''_1 , \ldots , v''_{n-1}
$$
are pairwise orthogonal and each of the vectors $u,u',u''$ is orthogonal to each of the above vectors.

In order to complete the proof we only need to show that there exist complex numbers $z',z''$ of modulus one such that $u=z' u'$ and $u=z''u''$.
We know that $\phi (U) \cap \phi (V)$ is a nontrivial subspace. Choose a nonzero vector $w \in \phi (U) \cap \phi (V)$. Then
\begin{equation}\label{boanebo}
w = \lambda u + \sum_{j=1}^{n-1} \mu_j v_j = \lambda' u' + \sum_{j=1}^{n-1} \mu'_j v'_j 
\end{equation}
and therefore
$$
\lambda u - \lambda' u' \in [v_1 , \ldots , v_{n-1}, v'_1 , \ldots , v'_{n-1}] .
$$
The left-hand side vector is orthogonal to the right-hand side vector space. Thus $\lambda u = \lambda' u'$. If $\lambda=0$ then $\lambda' = 0$ and then (\ref{boanebo}) would imply $w=0$, a contradiction. Thus, both $\lambda$'s are nonzero which means that 
$u=z' u'$ for some complex number $z'$ of modulus one. In the same way we see that
 $u=z''u''$ for some complex number $z''$ of modulus one.

\end{proof}

\section{Proofs}

This section is devoted to the proofs of our main results. The first one is trivial.

\bigskip

\noindent
{\sl Proof of Proposition \ref{lako}}. Clearly, $\phi$ preserves orthogonality, that is, for every pair $P,Q \in P_\infty (H)$ we have
$$
P\perp Q \iff \phi (P) \perp \phi (Q) .
$$
We claim that $\phi$ is injective. Indeed, if $P, Q \in P_\infty (H)$
are distinct, then there exists $R \in P_\infty (H)$ which is orthogonal to one of them but not orthogonal to the other one, say $R \perp P$ and $R\not\perp Q$.
Then $\phi( R) \perp \phi (P)$ and $\phi( R)\not\perp \phi (Q)$. It follows that $\phi(P) \not= \phi(Q)$,
and consequently, $\phi$ is injective. Hence, it is bijective. The statement is now a direct consequence of \cite[Theorem 1.2]{Sem}.
\enp

\bigskip
\noindent
{\sl Proof of Proposition \ref{manje}}. Since $H$ is infinite-dimensional we can identify $H$ with the orthogonal direct sum of two copies of $H$. Using this identification we can identify maps from $P_\infty (H)$ to
itself with maps from $P_\infty (H)$ to $P_\infty ( H \oplus H)$. 

Let $\varphi : P_\infty (H) \to P (H)$ be any map such that $0 \not= \varphi (P) \le P$ for every $P \in P_\infty (H)$.  Clearly, for every pair of projections $P,Q \in P_\infty (H)$
we have
$$
{\rm ma}\, (\varphi (P) , \varphi (Q) ) \ge {\rm ma} \, (P,Q).
$$
We define $\phi : P_\infty (H) \to P_\infty ( H \oplus H)$ by
$$
\phi (P) = \left[ \begin{matrix}  { P & 0 \cr 0 & \varphi (P) \cr } \end{matrix} \right] , \ \ \ P\in P_\infty (H). 
$$
Using Lemma \ref{padejmo} we verify that $\phi$ satisfies ${\rm ma}\, (\phi (P) , \phi (Q) ) = {\rm ma} \, (P,Q)$ for all $P,Q \in P_\infty (H)$. Since the only assumption on the map $\varphi$ is that
$0 \not= \varphi (P) \le P$ for every $P \in P_\infty (H)$ one can easily choose $\varphi$ in such a way that $\phi$ is not standard (for example, the fact that $\phi$ is standard implies $\varphi (P) \le \varphi (Q)$
for any pair of projections $P,Q \in P_\infty (H)$ satisfying $P \le Q$, and one can choose $\varphi$ in such a way that this condition is not fulfilled).
\enp

In order to prove Theorem \ref{main} we will distinguish four cases. 
The special case when $n=1$ together with the remark following the formulation of our main result is a non-bijective version of Wigner's theorem. For a very short proof we refer to \cite{Ge0}.

We will now formulate three results on maps preserving the minimal principal angle. Our main theorem is a straightforward consequence. Actually, the last two results below are much stronger than the corresponding versions of Theorem \ref{main}. One of the results below has been known before \cite{Pan} and the proofs of the other two rely heavily on some results of Blunck and Havlicek \cite{BlH}, and Pankov \cite{Pan}.

\begin{theorem}\label{main1}
Let $n$ be a positive integer, $n > 1$, and
$H$  a Hilbert space with $\dim H = 2n$. Assume that $\phi : P_n (H) \to P_n (H)$ is a map satisfying
$$
{\rm ma}\, (\phi (P ), \phi (Q)) = {\rm ma}\, (P,Q)
$$ 
for all $P,Q \in P_n (H)$. Then there exists a unitary or antiunitary operator $U : H \to H$ such that
$$
\phi (P) = UPU^\ast
$$
for every $P \in P_n (H)$, or
$$
\phi (P) = U(I-P)U^\ast
$$
for every $P \in P_n (H)$.
\end{theorem}

\begin{proof}
We will first show that $\phi$ is continuous. Let $(R_k) \subset P_n (H)$ be a sequence of projections of rank $n$ that converges to $P$ and denote by $Q = I-P$ the unique rank $n$ projection that is orthogonal to $P$. From
$$
\lim_{k \to \infty} R_k = P
$$
we deduce that ${\rm ma}\, (R_k , Q)$ tends to ${\pi \over 2}$. This is geometrically obvious. For a formal proof of this statement we notice that 
$$
\lim_{k \to \infty} (R_k -Q) = \lim_{k \to \infty} (R_k - (I-P)) = 2P-I.
$$
Because $2P-I$ is a unitary operator we see that the sequence of $(2n)$-tuples of singular values of $(R_k - Q)$ converges to the $(2n)$-tuple $(1,1,\ldots,1)$ and applying the fact the sines of non-zero principal angles between ${\rm Im}\, R_k$ and ${\rm Im}\, Q$
are exactly the non-zero singular values of the operator $R_k - Q$, each of them counted twice, we see that ${\rm ma}\, (R_k , Q)$ tends to ${\pi \over 2}$.
It follows that
$$
\lim_{n \to \infty} {\rm ma}\, (\phi (R_k ) ,\phi ( Q)) = {\pi \over 2}.
$$
Obviously, this yields that $\phi (R_k )$ converges to the unique projection of rank $n$ that is orthogonal to $\phi (Q)$. In other words,
$$
\lim_{n \to \infty} \phi (R_k) = \phi (P),
$$
yielding the continuity of $\phi$.

We further know that $\phi$ is injective.
Since $P_n (H)$ is a compact manifold and $\phi$ is continuous, $\phi (P_n (H))$ is also compact. 
On the other hand, the invariance of domain theorem ensures that $\phi (P_n (H)) \subset P_n (H)$ is open as well. 
Since $P_n (H)$ is connected, we conclude that $\phi$ is a bijective map from $P_n (H)$ onto itself. 

Let $P,Q \in P_n (H)$. Then ${\rm Im}\, P$ and ${\rm Im}\, Q$ are complemented subspaces of $H$, that is, $H = {\rm Im}\, P \oplus {\rm Im}\, Q$, if and only if ${\rm Im}\, P \cap {\rm Im}\, Q = \{ 0 \}$, and this is further equivalent to ${\rm ma}\, (P,Q) \not=0$. Thus, 
${\rm Im}\, P$ and ${\rm Im}\, Q$ are complemented if and only if ${\rm Im}\, \phi (P)$ and ${\rm Im}\, \phi (Q)$ are complemented. It follows from \cite[Theorem 4.4]{BlH} that
there exists a bijective semi-linear transformation $U\colon H\to H$ such that we have 
either
\begin{itemize}
\item ${\rm Im}\, \phi(P) = U ({\rm Im}\, P)$ for all $P\in P_n (H)$, or
\item${\rm Im}\, \phi(P) = U (({\rm Im}\, P)^\perp)$ for all $P\in P_n (H)$.
\end{itemize}
Since $\phi$ is continuous, $U$ must be continuous, and therefore, it is either linear or conjugate-linear map.

We start with the first possibility. Let $x,y \in H$ be orthogonal vectors. Then we can find orthogonal projections $P,Q = I-P \in P_n (H)$ such that $x$ belongs to
the image of $P$ while $y$ belongs to the image of $Q$. It follows from $\phi (P) \perp \phi (Q)$ that $Ux \perp Uy$. Assume now that $x,y \in H$ are orthogonal
and $\| x \| = \| y \|$. Then $x-y$ is orthogonal to $x+y$ and therefore $U(x-y) \perp U(x+y)$ implying that $\| U x \| = \| U y \|$. We conclude that for an arbitrary
pair of vectors $x,y \in H$ satisfying $\| x \| = \| y \|$ we have $\| U x \| = \| U y \|$. Indeed, take $z \in H$ with 
$\| x \| = \| y \| = \| z \|$ such that $z$ is orthogonal to both $x$ and $y$. Then  $\| U x \| = \| Uz \| = \| U y \|$, as desired.
It follows that $U$ is a scalar multiple of a unitary or an antiunitary operator. So, we may assume with no loss of generality that $U$ is either
a unitary or an antiunitary operator.

Thus, in the first case we have
$$
\phi (P) = UPU^\ast , \ \ \ P \in P_n (H),
$$
for some unitary or an antiunitary operator $U$. And in the second case we conclude in an almost the same way that there exists a unitary or an antiunitary operator $U$ on $H$
such that $\phi (P) = U(I-P)U^\ast$ for every $P \in P_n (H)$.
\end{proof}

The next case we will treat is that $2n < \dim H < \infty$. It has been known before that in order to get the desired conclusion in this special case we do not need to assume that all principal minimal angles are preserved. It is enough to assume that only the orthogonality is preserved. The result below is not new, see \cite[Theorem 2]{Pan}.

\begin{theorem}\label{main2}
Let $n$ be a positive integer, $n > 1$, and
$H$  a Hilbert space with $2n < \dim H < \infty$. Assume that $\phi : P_n (H) \to P_n (H)$ is a map such that for every pair $P,Q \in P_n (H)$ we have
$$
{\rm ma}\, (\phi (P ), \phi (Q)) =  {\pi \over 2} \iff  {\rm ma}\, (P,Q) =  {\pi \over 2}.
$$ 
Then there exists a unitary or antiunitary operator $U : H \to H$ such that
$$
\phi (P) = UPU^\ast
$$
for every $P \in P_n (H)$.
\end{theorem}

In the remaining case when $\dim H = \infty$ we will assume the preservation of only two extremal principal minimal angles.

\begin{theorem}\label{main3}
Let $H$ be an infinite-dimensional Hilbert space, $n > 1$ an integer, and $\phi : P_n (H) \to P_n (H)$ a map such that for every pair $P,Q \in P_n (H)$ we have
$$
{\rm ma}\, (\phi (P ), \phi (Q)) =  {\pi \over 2} \iff  {\rm ma}\, (P,Q) =  {\pi \over 2}
$$ 
and
$$
{\rm ma}\, (\phi (P ), \phi (Q)) =  0 \iff  {\rm ma}\, (P,Q) =  0.
$$ 
Then there exists a linear or conjugate-linear isometry $U : H \to H$ such that
$$
\phi (P) = UPU^\ast
$$
for every $P \in P_n (H)$.
\end{theorem}

\begin{proof}
If we consider elements of $P_n (H)$ as $n$-dimensional subspaces of $H$ then we easily see that all the assumptions of Lemma \ref{simulant} are satisfied. We define a new map $\psi : P_1 (H) \to P_1 (H)$ in the following way. For any unit vector $x \in H$ we can find $U,V \in P_n (H)$ such that 
$U \sharp V$ and $U \cap V = [x]$. Then, by Lemma \ref{simulant} we know that $\phi (U) \sharp \phi (V)$ and we define
$$
\psi ([x]) = \phi (U) \cap \phi (V).
$$ 
We first need to check that $\psi$ is well-defined. Assume that $U,V,W,Z \in P_n (H)$ are subspaces such that 
$$
U \sharp V \ \ \ {\rm and} \ \ \ U \cap V = [x]
$$
and
$$
W \sharp Z \ \ \ {\rm and} \ \ \ W \cap Z = [x].
$$
Because $H$ is infinite-dimensional we can easily find subspaces $L,M \in P_n (H)$ such that
\begin{itemize}
\item any two distinct elements of the set $\{ U,V,L,M \}$ are 1-orthogonal and the intersection of any two distinct elements of the set $\{ U,V,L,M \}$ is equal to $[x]$; and
\item any two distinct elements of the set $\{ W,Z,L,M \}$ are 1-orthogonal and the intersection of any two distinct elements of the set $\{ W,Z,L,M \}$ is equal to $[x]$.
\end{itemize}
It is now a straightforward consequence of 
Lemma \ref{simulant} that
$$
\phi (U) \cap \phi (V) = \phi (W) \cap \phi (Z),
$$
as desired.

It is trivial to see that for any unit vectors $x,y \in H$ and any subspace $U \in P_n (H)$ we have
$$
[x] \subset U \Rightarrow \psi ([x]) \subset \phi (U)
$$
and
$$
[x] \perp [y] \Rightarrow \psi ([x]) \perp \psi ([y]).
$$

Recall that subspaces $U,V \subset P_n (H)$ are said to be adjacent if $\dim (U \cap V) = n-1$, or equivalently, $\dim (U+V) = n+1$. We claim that if $U,V \in P_n (H)$ are adjacent, then $\phi (U)$ and $\phi (V)$ are adjacent. Indeed, if $U$ and $V$ are adjacent then we can find pairwise orthogonal unit vectors $x_1 , \ldots , x_{n-1} \in U \cap V$. It follows that $\psi ([x_1]), \ldots , \psi ([x_{n-1}])$ are pairwise orthogonal one-dimensional subspaces of $\phi (U) \cap \phi (V)$ implying that $\dim (\phi (U) \cap \phi (V) ) \ge n-1$. But we know that $\phi$ is injective and therefore $\phi (U) \not= \phi (V)$ which yields that
$\dim (\phi (U) \cap \phi (V) ) \le n-1$. Thus, $\phi (U)$ and $\phi (V)$ are adjacent. 

We complete the proof using Pankov's result \cite[Theorem 1]{Pan} stating that every orthogonality and adjacency preserving map $\phi : P_n (H) \to P_n (H)$ is induced by a linear or conjugate-linear isometry. 
\end{proof}

\bigskip
\bigskip
\noindent
There is no conflict of interest. The manuscript has no associated data.


\begin{thebibliography}{99}

\bibitem{BlH}A. Blunck and H. Havlicek, On bijections that preserve complementarity of subspaces, {\em Discrete Math.} {\bf 301} (2005), 46-56.

\bibitem{BJM} 
F. Botelho, J. Jamison, and L. Moln\' ar, 
Surjective isometries on Grassmann spaces,
{\em J. Funct. Anal.} {\bf 265} (2013), 2226-2238. 



\bibitem{Gal}
A. Gal\'antai,
Subspaces, angles and pairs of orthogonal projections,
{\em Linear Multilinear Algebra} {\bf 56} (2008), 227--260.

\bibitem{Ge0}G.P. Geh\' er, An elementary proof for the non-bijective version of Wigner's theorem, {\em Physics Letters A} {\bf 378} (2014), 2054-2057.

\bibitem{Geh}
G.P. Geh\' er, 
Wigner's theorem on Grassmann spaces,
{\em J. Funct. Anal.} {\bf 273} (2017), 2994--3001.

\bibitem{GeS}
 G.P. Geh\' er and P. \v Semrl, 
Isometries of Grassmann spaces,  
{\em J. Funct. Anal.} {\bf 270}  (2016), 1585-1601.

\bibitem{GeS2}
 G.P. Geh\' er and P. \v Semrl, 
Isometries of Grassmann spaces, II, {\em Adv. Math.} {\bf 332} (2018), 287-310. 


\bibitem{Mo}
L. Moln\'ar,
Transformations on the set of all $n$-dimensional subspaces of a Hilbert space preserving principal angles,
{\em Comm. Math. Phys.} {\bf 217} (2001), 409--421.


\bibitem{Mo2}
L. Moln\'ar,
Maps on the $n$-dimensional subspaces of a Hilbert space preserving principal angles,
{\em Proc. Amer. Math. Soc.} {\bf 136} (2008), 3205--3209.

\bibitem{Pan}
M. Pankov,
Geometric version of Wigner's theorem for Hilbert Grassmannians,
{\em J. Math. Anal. Appl.} {\bf 459} (2018), 135-144.


\bibitem{Sem} 
P. \v Semrl, 
Orthogonality preserving transformations on the set of $n$-dimensional subspaces of a Hilbert space,
{\em Illinois J. Math.} {\bf 48} (2004), 567--573.


\end{thebibliography}
\end{document}